\documentclass[article]{amsart}
\usepackage{amsmath}
\usepackage{amsthm}
\usepackage{amssymb}
\usepackage{amsthm}
\usepackage{bbm}
\usepackage{cite}
\usepackage{stmaryrd}
\usepackage{color}
\SetSymbolFont{stmry}{bold}{U}{stmry}{m}{n}
\usepackage{euscript}
\usepackage[arrow,curve,matrix,arc,2cell]{xy}
\UseAllTwocells
\usepackage[utf8]{inputenc}
\usepackage[unicode]{hyperref}
\usepackage{slashed}
\DeclareFontFamily{U}{rsfs}{} 
\DeclareFontShape{U}{rsfs}{n}{it}{<->
rsfs10}{} \DeclareSymbolFont{mscr}{U}{rsfs}{n}{it}
\DeclareSymbolFontAlphabet{\scr}{mscr}
\def\mathscr{\scr}
\newcommand{\dbar}{\bar{\partial}}

\def\C{\mathbb{C}}
\def\R{\mathbb{R}}
\def\Z{\mathbb{Z}}

\theoremstyle{plain}
\newtheorem{theorem}{Theorem}[section]
\newtheorem{proposition}[theorem]{Proposition}

\newtheorem{Corollary}[theorem]{Corollary}

\theoremstyle{definition}
\newtheorem{definition}[theorem]{Definition}
\newtheorem{remark}[theorem]{Remark}

\makeatletter
\@addtoreset{equation}{section}
\makeatother
\begin{document}

%%%%%%%%%%%%%%%%%%%%%%%%%%%%%%%%%%%%%%%%%%%%%%%%%%%%%%%%%%%%%%%%%%%%%%%%%%%%%%%%%%%%%%%%%%%%%%%%%%%%%
%%%%%%%%%%%%%%%%%%%%%%%%%%%%%%%%%%%%%%%%%%%%%%%%%%%%%%%%%%%%%%%%%%%%%%%%%%%%%%%%%%%%%%%%%%%%%%%%%%%%%
 
\title[Kapustin--Witten and nonabelian Hodge]{The Kapustin--Witten equations and nonabelian Hodge theory} 

%%%%%%%%%%%%%%%%%%%%%%%%%%%%%%%%%%%%%%%%%%%%%%%%%%%%%%%%%%%%%%%%%%%%%%%%%%%%%%%%%%%%%%%%%%%%%%%%%%%%%
\author{Chih-Chung Liu}
\address{C.-C. Liu: Department of Mathematics, National Cheng-Kung University, Tainan, 70101, Taiwan}
\email{cliu@mail.ncku.edu.tw}

\author{Steven Rayan}
\address{S. Rayan: Centre for Quantum Topology and Its Applications (quanTA) and Department of Mathematics \& Statistics, University of Saskatchewan, Saskatoon, SK, S7N 5E6, Canada}
\email{rayan@math.usask.ca}

\author{Yuuji Tanaka}
\address{Y. Tanaka: Department of Mathematics, Faculty of Science, Kyoto University
Kitashirakawa Oiwake-cho, Sakyo-ku, Kyoto, 606-8502, Japan}
\email{y-tanaka@math.kyoto-u.ac.jp}

%%%%%%%%%%%%%%%%%%%%%%%%%%%%%%%%%%%%%%%%%%%%%%%%%%%%%%%%%%%%%%%%%%%%%%%%%%%%%%%%%%%%%%%%%%%%%%%%%%%%%%

\maketitle

%%%%%%%%%%%%%%%%%%%%%%%%%%%%%%%%%%%%%%%%%%%%%%%%%%%%%%%%%%%%%%%%%%%%%%%%%%%%%%%%%%%%%%%%%%%%%%%%%%%%%
\begin{abstract}
Arising from a topological twist of $\mathcal{N}=4$ super Yang--Mills theory are the Kapustin--Witten equations, a family of gauge-theoretic equations on a four-manifold parametrised by $t\in\mathbb{P}^1$. The parameter corresponds to a linear combination of two super charges in the twist.  When $t=0$ and the four-manifold is a compact K\"ahler surface, the equations become the Simpson equations, which was originally studied by Hitchin on a compact Riemann surface, as demonstrated independently in works of Nakajima and the third-named author.  At the same time, there is a notion of $\lambda$-connection in the nonabelian Hodge theory of Donaldson--Corlette--Hitchin--Simpson in which $\lambda$ is also valued in $\mathbb{P}^1$.  Varying $\lambda$ interpolates between the moduli space of semistable Higgs sheaves with vanishing Chern classes on a smooth projective variety (at $\lambda=0$) and the moduli space of semisimple local systems on the same variety (at $\lambda=1$) in the twistor space.  In this article, we utilise the correspondence furnished by nonabelian Hodge theory to describe a relation between the moduli spaces of solutions to the equations by Kapustin and Witten at $t=0$ and $t \in \R \setminus \{ 0 \}$ on a smooth, compact K\"ahler surface. We then provide supporting evidence for a more general form of this relation on a smooth, closed four-manifold by computing its expected dimension of the moduli space for each of $t=0$ and $t \in \R \setminus \{ 0 \}$.
\end{abstract}
%%%%%%%%%%%%%%%%%%%%%%%%%%%%%%%%%%%%%%%%%%%%%%%%%%%%%%%%%%%%%%%%%%%%%%%%%%%%%%%%%%%%%%%%%%%%%%%%%%%%%%

%\renewcommand{\thefootnote}{\fnsymbol{footnote}}
%\footnote[0]{2000\textit{ Mathematics Subject Classification}.
% Primary 53C07 ; Secondary 49Q15, 53C15 ,  57R57 , 81T13.}

%%%%%%%%%%%%%%%%%%%%%%%%%%%%%%%%%%%%%%%%%%%%%%%%%%%%%%%%%%%%%%%%%%%%%%%%%%%%%%%%%%%%%%%%%%%%%%%%%%%%%%

%\markboth{}
%{}

%%%%%%%%%%%%%%%%%%%%%%%%%%%%%%%%%%%%%%%%%%%%%%%%%%%%%%%%%%%%%%%%%%%%%%%%%%%%%%%%%%%%%%%%%%%%%%%%%%%%%%

\setcounter{tocdepth}{2}
\tableofcontents

\section{Introduction}

In this article, we consider a family of equations introduced by Kapustin and Witten \cite{KaWi} on a closed four-manifold with the aim of describing their relationship to a higher-dimensional instance of the nonabelian Hodge correspondence \cite{Simp2,Simp4}.  To set the stage for it, we let $G$ be a compact Lie group and $X$ a closed, oriented, smooth four-manifold, with $P \to X$ a principal $G$-bundle over $X$.  We fix a Riemannian metric on $X$. 
Furthermore, let $A$ be a connection on $P$, and let $\mathfrak{a}$ be a smooth section of the bundle $\text{ad} (P) \otimes \Lambda^1_{X}$, where $\Lambda^1_X := T^*X$. We say that such data $(A,\mathfrak a)$ is solution to the ($G$\emph{-)Kapustin-Witten equations} on $X$ for $t\in\mathbb P^1$ if they satisfy
\begin{gather}
( F_{A} - [\mathfrak{a} \wedge \mathfrak{a} ] + t d_{A} \mathfrak{a} )^{+} =0,  
\label{eq:pKW1i} \\
( F_{A} - [\mathfrak{a} \wedge \mathfrak{a} ] - t^{-1} d_{A} \mathfrak{a} )^{-} =0, 
\label{eq:pKW2i} \\
d_{A}^{*} \mathfrak{a} =0 , 
\label{eq:pKW3i}  
\end{gather}
where $\pm$ in the first two conditions are the projections to the self-dual and anti-self-dual parts of $\mbox{ad}(P)$-valued two-forms respectively, $F_A$ is the usual curvature of $A$, and $d_{A}^{*}$ is the formal adjoint of the covariant derivative $d_{A}$.  Here, $t$ is a parameter that indexes the family of equations \eqref{eq:pKW1i}--\eqref{eq:pKW3i}.  Since the equations \eqref{eq:pKW1i}--\eqref{eq:pKW3i} form an overdetermined system for $t \not\in \R  \cup \{ \infty \}$ and the equations for $t = \infty$ are the orientation-reversed equations for $t=0$, we restrict ourselves to the case $t \in \R$ in this article.

Taubes \cite{Taub2} examined the asymptotic behaviour of the system of the equations \eqref{eq:pKW1i}--\eqref{eq:pKW3i} for $t=0$ as a means of compactifying the moduli space of solutions. One may ask about the algebro-geometric counterpart of this construction, as it may yield a different but definitely useful approach to the problem as in Donaldson's theory or other gauge-theoretic moduli problems in even dimensions. To facilitate this question, it is appropriate to let $X$ be a compact K\"ahler surface.  In this case, the $t=0$ equations reduce to the Simpson equations (cf. \cite{Naka}, \cite{Tana}), which are a generalisation of the Hitchin equations \cite{Hitc}.  Thus, the moduli space for $t =0$ is, algebro-geometrically, that of semistable (or polystable) Higgs sheaves by a result of Simpson \cite{Simp1}.

The moduli space of semistable Higgs sheaves is one of the main objects in nonabelian Hodge theory. 
The other one is the moduli of semisimple local systems, that is, the space of representations of the fundamental group $\pi_1(X)$ via the Riemann--Hilbert correspondence (see e.g. \cite{GaRa}). 
The classical nonabelian Hodge theorem \cite{Simp2, Simp4, Simp5, Simp6} states that the moduli space of semistable Higgs sheaves with vanishing Chern classes and that of semisimple local systems are real-analytically isomorphic.

Furthermore, this can be realised by means of $\lambda$-connections, which are connections parametrized by $\mathbb{P}^1$ that interpolate between these two moduli spaces at $\lambda=0$ and $\lambda=1$, respectively, via the twistor space (see \cite{Simp2, Simp5, Simp6, Simp7}). 
The interpolation can be identified with hyperk\"ahler rotation, as the distinct complex structures that endow the underlying smooth manifold with the structure of the moduli space of semistable Higgs sheaves with vanishing Chern classes and that of semisimple local systems, respectively, are two complex structures in a hyperk\"ahler triple.  These structures are two points on a copy of $S^2$, which can be identified with the twistor line coordinatized by $\lambda$ (cf. \cite{Simp2,GaRa} for more on this point of view).

One may observe perhaps the situation in the Kapustin--Witten theory is similar to that of this nonabelian Hodge theory, especially on a compact K\"{a}hler surface, where the semistable Higgs moduli space is located at the origin of the parameter space $\mathbb{P}^1$ as mentioned above.  
Then one may wonder what the shape of the Kapustin--Witten equations on a compact K\"{a}hler surface for $t \neq 0$ would be and whether it could be related to nonabelian Hodge theory. 
In this article, we give answers to the questions. 
To describe them, let $X$ be a smooth, compact K\"{a}hler surface, and let $E$ be a Hermitian vector bundle on $X$ with Hermitian metric $k$. 
We denote by $\mathcal{A}_{(E,k)}$ the space of $k$-unitary connections on $E$ and by $\text{End} (E)_0 = \text{End} (E,k)_0$ the bundle of trace-free skew-Hermitian endomorphisms of $E$ on $X$. 
We decompose $T^* X \otimes \C$ into $\Lambda^{1,0}_X \oplus \Lambda^{0,1}_X$ by the complex structure of $X$.  
In Section \ref{Sec:4} of this article, we prove the following: 
\begin{proposition}[Proposition \ref{prop:tnzKW}]
Let $X$ and $E$ be as above.  
Suppose $c_1 (E) = c_2 (E)=0$. Then, the Kapustin--Witten equations for $t \in \R \setminus \{ 0 \}$ reduces to the following equations on $X$ for $(A, \phi) \in \mathcal{A}_{(E,k)} \times \Gamma ( \emph{End} (E)_0 \otimes \Lambda^{1,0}_X )$: 
\begin{gather}
F_{A}^{0,2}= 0, \, \bar{\partial}_{A} \phi  =  0, 
\, \phi \wedge \phi =0,
\label{eq:tnzKWJ1i}  \\
F_{A}^{1,1} + [ \phi \wedge \phi^{*} ] = 0, \, \partial_A \phi = 0 . 
\label{eq:tnzKWJ2i} 
\end{gather}
\end{proposition}
Note that these are independent of $t$ as originally observed by Gagliardo and Uhlenbeck for closed four-manifolds \cite{GaUh}.

Surprisingly, the system \eqref{eq:tnzKWJ1i}, \eqref{eq:tnzKWJ2i} is equivalent to the one defining {\it pluri-harmonic metrics} on a Higgs bundle on $X$ in nonabelian Hodge theory --- our Section \ref{Sec:3} includes a brief account on pluri-harmonic metrics in the context of nonabelian Hodge theory.  Consequently, the moduli space of solutions to \eqref{eq:tnzKWJ1i}, \eqref{eq:tnzKWJ2i} can be naturally identified with that of semisimple local systems. Moreover, via nonabelian Hodge theory, we obtain the following:

\begin{theorem}[Theorem \ref{th:4}]
Let $X$ be a smooth, compact K\"{a}hler surface $X$, and let $E$ be a Hermitian vector bundle on $X$ with $c_1(E)=c_2 (E) =0$. 
Then the moduli space of solutions to the $t=0$ Kapustin--Witten equations \eqref{eq:tzKWJ1}, \eqref{eq:tzKWJ2} on $E$ over $X$ is real-analytically isomorphic (away from any singular loci) to the moduli space of solutions to the $t \neq 0$ Kapustin--Witten equations \eqref{eq:tnzKWJ1}, \eqref{eq:tnzKWJ2} on $E$ over $X$. 
\end{theorem}

We expect that there would be a similar correspondence between the moduli spaces on a compact smooth four-manifold as well (in the absence of the K\"ahler geometry). However, there are a number of technical issues that must be overcome in the more general case, such as issues of smoothness and transversality of the moduli spaces. That said, we do prove:

\begin{theorem}[Corollary \ref{cor:vd}]
Consider a principal $G$-bundle $P$ over a closed, oriented, smooth four-manifold $X$, where $G$ is $SU(2)$ or $SO(3)$.
Then the expected dimensions of moduli spaces of solutions to the Kapustin--Witten equations for $t=0$ and $t \neq 0$ over $X$ coincide under a particular flatness condition (the vanishing of $k(G)$, as defined in Definition \ref{Defn:k(G)}).
\end{theorem}

This result can be taken as evidence in support of a more general correspondence, and it is then definitely interesting to study deformation spaces of solutions to the Kapustin--Witten equations for both $t=0$ and $t \neq 0$ on a closed four-manifold in an appropriate sense; and to compare the moduli spaces of them to obtain a certain form of correspondence indicated by nonabelian Hodge theory. A promising language for this is available in the derived geometry of Joyce \cite{Joyc1}, \cite{Joyc2}, given that the objects are topological in nature and so the standard metric perturbation method in gauge theory would not work nicely for them.  
Note that Pantev--To\"en--Vaqui\'e--Vezzosi \cite{PTVV} gives a $(-2)$-shifted symplectic structure to the moduli stack of perfect complexes of local systems on a compact, oriented, topological four-manifold $X$, which lives in the $t \neq 0$ side of the Kapustin--Witten theory. 
The orientation problem for the moduli space of solutions to the Kapustin--Witten equations for $t=0$ was solved in \cite[Theorem~4.9]{JTU}, and it is just straightforward to see that the same statement holds for the $t \neq 0$ case.

Lastly, one can study the $G$-Kapustin--Witten equations for the case of reductive, non-compact groups $G_{\C}$ such as $G_{\C}=GL(r,\C), SL(r,\C)$, and $PGL(r,\C)$. Given what is observed for Hitchin systems on curves \cite{HaTh,KaWi,DoPa}, it is compelling to ask whether moduli spaces of Kapustin--Witten solutions admit a fibration, in the sense of the Hitchin fibration \cite{Hitc2} (cf. \cite{Simp6, ChNg} for higher dimensions), that respects Langlands duality. In particular, one might compare the moduli spaces for $G_{\C}=SL(r,\C)$ moduli space and ${}^L G_{\C} = PGL(r, \C)$ and ask whether they are ``dual'' with respect to this fibration. Related to this hypothesis is the more general question of whether the moduli spaces Calabi--Yau in some appropriate sense (or even hyperk\"ahler), in which case the duality may have a physical interpretation (although we do not speculate on this).  To answer these questions, one must first define the appropriate fibration on the total spaces carefully. We leave this for future work.

To summarize, the structure of this article is as follows:
\begin{itemize}
\item Section \ref{Sec:2} describes basic properties of solutions to the Kapustin--Witten equations on closed four-manifolds. 
\item Section \ref{Sec:3} is a brief overview of classical nonabelian Hodge theory with a view to the application we have in mind. 
\item We study the Kapustin--Witten equations on a compact K\"{a}hler surface and describe their relationship to the nonabelian Hodge theory in Section \ref{Sec:4}. 
\item In Section \ref{Sec:5}, we compute the expected dimensions of the moduli spaces of solutions to both the $t=0$ and $t \neq 0$ Kapustin--Witten equations on a general smooth, closed four-manifold.   
\end{itemize}

\section{A family of equations in Kapustin--Witten theory}
\label{Sec:2}

Let $X$ be a closed, oriented, smooth four-manifold, and let $P \to X$ be a principal $G$-bundle over $X$, where $G$ is a compact Lie group.  
Fix a Riemannian metric on $X$. 
It induces the Hodge star operator $* : \Lambda^{p} T^* X \to \Lambda^{4-p} T^* X$ for $p = 0, \dots , 4$. Since $*^2 =1$ on $\Lambda^2_X := \Lambda^2 T^* X$, $\Lambda^2_X$ decomposes into  {\it self-dual} part  $\Lambda^+_X$ and {\it anti-self-dual} part $\Lambda^-_X$, corresponding to the eigenvalues $\pm 1$ respectively.  Accordingly, we have the decomposition $\text{ad}(P) \otimes \Lambda^2_X = \text{ad}(P) \otimes \Lambda^+_X \oplus \text{ad}(P) \otimes \Lambda^-_X$ .

From a topological twist of $\mathcal{N}=4$ super Yang--Mills theory, Kapustin and Witten \cite{KaWi} introduced 
the following family of equations parametrized by $t \in \mathbb{P}^1$ for a pair $(A, \mathfrak{a})$ consisting of a connection $A$ on $P$ and a section $\mathfrak{a}$ of the bundle $\text{ad} (P) \otimes \Lambda^1_{X}$:   
\begin{gather}
( F_{A} - [\mathfrak{a} \wedge \mathfrak{a} ] + t d_{A} \mathfrak{a} )^{+} =0,  
\label{eq:pKW1} \\
( F_{A} - [\mathfrak{a} \wedge \mathfrak{a} ] - t^{-1} d_{A} \mathfrak{a} )^{-} =0, 
\label{eq:pKW2} \\
d_{A}^{*} \mathfrak{a} =0 , 
\label{eq:pKW3}  
\end{gather}
where the superscripts $\pm$ denote the respective projections to the self-dual part $\Gamma ( \text{ad} (P) \otimes \Lambda^+_X )$ and the anti-self-dual one $\Gamma ( \text{ad} (P) \otimes \Lambda^-_X )$ respectively, and $d_{A}^{*}$ is the formal adjoint of the covariant derivative $d_{A}$ with respect to the $L^2$-inner product. 
As mentioned in Introduction, we restrict ourselves to the case $t \in \R $ since the equations are overdetermined for $t \not\in \R \cup \{ \infty \}$ and the equations for $t=\infty$ is the orientation-reversed one for $t=0$.  
For the case $t \in \R$, the equations \eqref{eq:pKW1}--\eqref{eq:pKW3} with a gauge fixing equation form an elliptic system.

\begin{remark}
The physical parameter on which the Kapustin--Witten theory depends is a combination of this twisting parameter $t$ and the coupling parameter coming from the coupling constant and the theta-angle in the theory.  See \S 3.5 of \cite{KaWi} for more details. 
\end{remark}

Kapustin and Witten \cite{KaWi} discovered an astonishing relationship between the geometric Langlands programme on a curve and the $S$-duality conjecture for $\mathcal{N}=4$ super Yang--Mills theory in dimension four via the topologically-twisted theory compactifed on the curve, which was further generalised to the ramified case by Gukov--Witten \cite{GuWi}.
It is worth noting that this relationship interacts with the hyperk\"ahler structure on the moduli space of semistable Higgs bundles over the curve. This interaction picks out distinguished triple branes (branes that are either symplectic or complex with regards to each complex structure) in the moduli space that are transformed under mirror symmetry, which itself is a form of global Langlands duality (see also Hausel--Thaddeus \cite{HaTh} and Donagi--Pantev \cite{DoPa}).  The classification of triple branes in semistable Higgs bundle moduli spaces is explored in \cite{BaSc1,BaSc2,BGPH,FGOP} for instance, and has been furthermore linked to Nakajima quiver varieties in \cite{FJM,RaSc}.   There are also emerging connections between Kapustin--Witten theory and multiplicative Hitchin systems that are explored in \cite{ElPe}.

In this article, however, the underlying four-manifold is not necessarily the product of two Riemann surfaces, and we describe nonabelian-Hodge-theoretic aspect of the theory on complex surfaces, rather than on curves, and discuss a further possible realisation of it on closed four-manifolds.

For $t =0$ the equations \eqref{eq:pKW1}--\eqref{eq:pKW3} become 
\begin{equation}
F_{A}^{+} - [\mathfrak{a} \wedge \mathfrak{a} ]^{+} =0 , \,  ( d_{A} \mathfrak{a} )^{-} =0, \,  d_{A}^{*} \mathfrak{a}=0 . 
\label{eq:tzKW} 
\end{equation}
Solutions to the above equations \eqref{eq:tzKW}, especially their asymptotic behaviours, were intensively studied by Taubes \cite{Taub2}. 

We now require the following definition:

\begin{definition}\label{Defn:k(G)} For a principal $G$-bundle $P$ over $X$, we define the integer $k(G)$ to be $c_2 ( P \otimes \C ) [X]$ if $G$ is $SU(2)$ and by $- \frac{1}{4} p_1 (P) [X]$ if $G$ is $SO(3)$.
\end{definition}

Gagliardo and Uhlenbeck \cite[Corollary 3.3]{GaUh} show that the equations \eqref{eq:pKW1}--\eqref{eq:pKW3} have a solution if and only if $k(G)=0$, and they are independent of $t$ for $t \in \R \setminus \{ 0 \}$. 
Namely, they prove:

\begin{theorem}[Gagliardo--Uhlenbeck]
Assume that the structure group $G$ is $SU(2)$ or $SO(3)$. 
Then the equations \eqref{eq:pKW1}--\eqref{eq:pKW3} have a solution if and only if $k(G)=0$. 
Furthermore, a pair $(A, \mathfrak{a})$ satisfies the equations \eqref{eq:pKW1}--\eqref{eq:pKW3} with $t \in \R \setminus \{ 0\}$ if and only if it satisfies the following equations: 
\begin{gather}
F_{A} - [\mathfrak{a} \wedge \mathfrak{a} ] =0, \, 
d_{A} \mathfrak{a} = d_{A}^{*} \mathfrak{a} =0 .  
\label{eq:tnzKW}
\end{gather}
\label{thm:GU}
\end{theorem}

The clever trick used by Gagliardo and Uhlenbeck to prove the above theorem is to consider the curvature $F_{A + i \mathfrak{a}}$ of the ``complexified'' connection $A + i \mathfrak{a}$ out of a pair $(A, \mathfrak{a})$ which satisfies the Kapustin--Witten equations \eqref{eq:pKW1}--\eqref{eq:pKW3} with $t \in \R \setminus \{ 0\}$ and the use of their fascinating energy identity for it (see \cite[\S 3]{GaUh} for more details). 

Similar equations to \eqref{eq:tnzKW} on smooth three-manifolds were studied by Taubes \cite{Taub1} in the process of analytically describing the {\it  Morgan--Shalen compactification} of $SL(2, \C)$-character varieties of three-manifolds.

\begin{remark}
The above $t$-independency is not true when $X$ is not compact (see e.g. Gaiotto--Witten \cite{GaWi} which develops the programme by Witten \cite{Witt} aiming at ``categorifying'' the knot invariants by using the system of the equations \eqref{eq:pKW1}--\eqref{eq:pKW3} on a four-manifold with boundary). 
\end{remark}

\section{Nonabelian Hodge theory}
\label{Sec:3}

For general references on classical nonabelian Hodge theory, we refer the reader to \cite{GaRa, Huan, Moch2, Simp2,Simp4, Simp6, Simp7}. In order to employ the necessary elements of the correspondence, we will need to introduce some algebro-geometric language.  Also, we will take a more general point of view in which $X$ is not necessarily of real dimension four, at least initially.  As such, we let $X$ be a smooth, compact K\"{a}hler manifold of complex dimension $n$ with K\"ahler form $\omega$.  Let its holomorphic cotangent bundle be denoted by $\Omega_{X}^{1}$ and let $E$ be a smooth complex vector bundle on $X$.  We are interested in two types of structures on $E$: \emph{Higgs structures} and \emph{flat structures}.  The former structure is that of a Higgs bundle while the latter is that of a flat connection.  The former necessitates that $E$ is Hermitian while the latter is not necessarily unitary with respect to any Hermitian metric.

\subsection{Higgs bundles and Hermitian-Yang--Mills metrics}
\label{Sec:3.1}

Suppose now that $\dbar_E$ is a holomorphic structure on $E$. Then, a $\mathcal O_X$-linear, $\dbar_E$-holomorphic map $\theta : E \to E \otimes \Omega_{X}^1$ is called a \emph{Higgs field} for $(E,\dbar_E)$.  When the integrability condition $\theta\wedge\theta=0$ is satisfied, we say that the triple $(E , \dbar_{E}, \theta)$ is a {\it Higgs bundle} on $X$.

We define a notion of {\it degree} for $E$ by $\deg (E) := \int_X c_1 (E) \cdot \omega^{n-1}$. Following from this, the {\it slope} of $E$ is$$\mu(E):=\deg (E) /\text{rank}(E).$$We will use $\mathcal E$ to refer to the locally-free sheaf of holomorphic sections of $(E,\dbar_E)$.  We define the rank, degree, and slope of the Higgs bundle $(E,\dbar_E,\theta)$ to be those of the underlying vector bundle $E$ (or sheaf $\mathcal E$), respectively.  With these ideas in place, we can define \emph{stability}:

\begin{definition}
We say $(E,\dbar_E,\theta)$ is {\it semistable} if, for each nonzero coherent subsheaf $\mathcal F\subsetneq \mathcal E$ with the property $\theta(\mathcal F)\subseteq\mathcal F\otimes\Omega^1_X$ (i.e. with that property that it is $\theta$-invariant), then we have$$\mu (\mathcal F)\leq \mu (\mathcal E).$$We call $(E,\dbar_E,\theta)$ {\it stable} if we replace $\leq$ with $<$ in the preceding sentence.
\label{Def:stable}
\end{definition}

\begin{definition}
We call $(E,\dbar_E,\theta)$ {\it polystable} if it is a direct sum of stable Higgs bundles of the same slope.
\label{Def:poly}
\end{definition}

\begin{remark}
One needs to consider Gieseker-semistable Higgs sheaves in order to set up the moduli problem properly on a surface or a higher-dimensional variety. However, we do not do so, since our focus here is on the correspondence between analytic objects and algebro-geometric ones. 
We restrict attention in this article to bundles (locally-free sheaves) rather than more general coherent sheaves because (a) the Kapustin--Witten equations are naturally written on a smooth bundle and (b) the ensuing discussion of the existence of special metrics is most natural within the setting of bundles.
\end{remark}

Now, given a Higgs bundle $(E,\dbar_{E} ,\theta)$ on $X$, consider a Hermitian metric $h$ on $E$. The combination of this data induces a $(1,0)$-connection $\partial_{E, h}$ so that $D_{\bar{\partial}_{E}, h}^1 := \bar{\partial}_{E}+\partial_{E,h}$ is a unique $h$-unitary connection on $E$ compatible with $\bar{\partial}_{E}$, which is called the {\it Chern connection} associated to $(\dbar_E,h)$.  
Denote by $\theta^{*,h}$ the adjoint of the Higgs field with respect to the Hermitian metric $h$ on $E$.

\begin{definition}
A Hermitian metric $h$ on a Higgs bundle $(E , \dbar_{E} , \theta)$ is said to be a {\it Hermitian-Yang--Mills} one if $\Lambda 
( F_h + [ \theta , \theta^{*, h} ] )= c \mbox{Id}_E$, where $\Lambda$ is the $L^2$-adjoint of $\wedge \omega$, $F_h$ is the curvature of $D_{\bar{\partial}_{E}, h}^1$ and a constant $c$ is given by $- \frac{2 \pi i}{r (n-1)! \int_{X} \omega^n} \cdot \text{deg} (E)$. 
\label{Def:HYM}
\end{definition}

Simpson \cite{Simp1} proves the following by using the heat equation method of Uhlenbeck--Yau in the stable bundle case \cite{UhYa}.

\begin{theorem}[Simpson]
A Higgs bundle admits a Hermitian-Yang--Mills metric if and only if it is polystable. 
\end{theorem}

\subsection{Flat bundles and harmonic metrics}

The other primary object in nonabelian Hodge theory is a {\it flat bundle} on $X$, namely,  
a pair $(E, \nabla)$ consisting of the smooth complex vector bundle $E$ and a flat $\mbox{GL}(r,\mathbb C)$-connection $\nabla$ on $E$, where $r$ is the rank of $E$. 
This is equivalent to considering a representation of the fundamental group of $X$ in $\mbox{GL}(r,\mathbb C)$ via the Riemann--Hilbert correspondence. 
We say $(E, \nabla)$ is {\it irreducible} if $E$ has no non-trivial $\nabla$-invariant subbundle.

\begin{definition}
We call $(E, \nabla)$ {\it semisimple} if it is a direct sum of irreducible flat bundles. 
\end{definition}

Fix $x \in X$. Then a flat bundle $(E, \nabla)$ gives a monodromy representation $\rho : \pi_1 (X, x) \to \mbox{GL }( E|_{x} ,\C)$. 
Denote by $\pi : \tilde{X} \to X$ the universal cover of $X$. 
We fix a point $\tilde{x} \in \tilde{X}$ with $\pi ( \tilde{x} ) = x$. 
We consider a Hermitian metric on $E$ and denote it by $h$. 
Then it gives a $\rho$-equivariant map $u_h : \tilde{X} \to GL (n , \C) / U(n)$.

\begin{definition}
A Hermitian metric $h$ on $E$ is said to be a {\it harmonic metric} if the induced $\rho$-equivariant map $u_h : \tilde{X} \to GL (n , \C) / U(n)$ is a harmonic one. 
\end{definition}

The following is attributed to Corlette \cite{Corl}, Donaldson \cite{Dona}, Jost--Yau \cite{JoYa} and Labourie \cite{Labo}, it is proved by the heat equation method pioneered by Eells--Sampson \cite{EeSa}.

\begin{theorem}
A flat bundle $(E, \nabla )$ admits a harmonic metric if and only if it is semisimple. 
\end{theorem}

\subsection{Harmonic bundles}

Let $(E, \bar{\partial}_{E}, \theta)$ be a Higgs bundle on $X$. 
Consider a Hermitian metric $h$ on $E$. 
This induces a $(1,0)$-connection $\partial_{E,h}$ so that $D_{\bar{\partial}_{E}, h}^1 := \dbar_{E} + \partial_{E,h}$ is a unique $h$-unitary connection on $E$ compatible with $\bar{\partial}_{E}$. This also defines $\theta^{*,h} \in \Omega^{0,1} (\text{End} (E) ) $ as above.  
Set $\mathbb{D}_{h}^1 := D_{\bar{\partial}_{E}, h}^1 + \theta + \theta^{*,h}$,  this is a connection on $E$.

\begin{definition}
Let  $(E , \dbar_{E} , \theta)$ be a Higgs bundle.  
A Hermitian metric $h$ on $E$ is said to be a {\it pluri-harmonic metric} on $(E , \dbar_{E} , \theta)$,  
if $\mathbb{D}_{h}^1$ is flat, i.e. $\mathbb{D}_{h}^1 \circ \mathbb{D}_{h}^1 =0$. 
We call such a quadruplet $(E , \dbar_{E}, \theta, h )$ a {\it harmonic bundle}.    
\end{definition}

On the other hand, the notion of pluri-harmonic metrics and harmonic bundles are defined from a flat bundle. 
Let $(E, \nabla)$ be a vector bundle $E$ on $X$ with flat connection $\nabla$. 
Suppose we are given a Hermitian metric $h$ on $E$. 
Then $\nabla$ uniquely decomposes into $D_{h}^0 + \Theta_h$, where $D_{h}^0$ is a $h$-unitary connection and $\Theta_{h}$ is the self-adjoint part. 
Decompose this $D_{h}^0$ into the (1,0) and (0,1) parts as $D_h^0 = \partial_{h} + \bar{\partial}_{h}$. 
Define $\theta_h$ and $\theta_h^{*}$ as the (1,0) and (0,1) parts of $\Theta_{h}$.

\begin{definition}
We call a Hermitian metric $h$ on $E$ a {\it pluri-harmonic metric} on $(E, \nabla)$, if $\mathbb{D}_{h}^0 := \bar{\partial}_{h} + \theta_h$ satisfies $\mathbb{D}_{h}^0 \circ \mathbb{D}_{h}^0 =0$, namely, $(E, \bar{\partial}_{E}, \theta_h )$ is a Higgs bundle. 
We also say $(E, \nabla , h)$ a {\it harmonic bundle}. 
\end{definition}

The object $\mathbb{D}^1_h$ introduced in the above is a connection (not necessarily flat) constructed from a Higgs bundle, while the object $\mathbb{D}^0_h$ is a Higgs bundle structure (not necessarily integrable) constructed from a flat connection. The question is whether we can recover one type of structure from the other and when these structures are flat (respectively, integrable).  The answers are in fact related.  When we start with a Higgs bundle $(E,\dbar_{E} ,\theta)$ for which $\theta\wedge\theta=0$, and if there exists a Hermitian metric $h$ such that $\mathbb{D}^1_h \circ \mathbb{D}^1_h=0$, then constructing $\mathbb{D}^0_h$ from $\mathbb{D}^1_h$ recovers the Higgs bundle $(E,\dbar_{E} ,\theta)$ and $\mathbb{D}^1_h \circ\mathbb{D}^1_h =0$.  Conversely, when we start with a flat connection $(E, \nabla)$ for which $\nabla^2=0$, and if there exists a Hermitian metric $h$ such that $\mathbb{D}^0_h \circ \mathbb{D}^0_h=0$, then constructing $\mathbb{D}^1_h$ from $\mathbb{D}^0_h$ recovers $(E,\nabla)$ and $\mathbb{D}^0_h \circ \mathbb{D}^0_h =0$.  Thus we are searching for the existence of a Hermitian metric $h$ that intertwines Higgs and flat structures in this way.

\subsection{Nonabelian Hodge correspondence.}

We briefly recall the notion of $\lambda$-connection introduced by Deligne and Simpson. 
It interpolates the polystable Higgs bundles with vanishing Chern classes and semisimple flat ones via the twistor space intrinsically attached to them (see e.g. \cite{Simp2, Simp7} for more details).

\begin{definition}
Let $\lambda \in \C$. 
A complex linear operator $\mathbb{D}^{\lambda} : E \to E \otimes \Omega^1_X$ is said to be a {\it $\lambda$-connection} if it satisfies 
$$ \mathbb{D}^{\lambda} (f \cdot s) = 
f  \mathbb{D}^{\lambda} (s) + \lambda s \otimes df ,$$ 
where $f$ is a function of $X$ and $s$ is a section of $E$. 
This extends to $\mathbb{D}^{\lambda} : E \otimes \Omega^{p}_{X} \to E \otimes \Omega^{p+1}_{X}$.  
It is said to be {\it flat} if $\mathbb{D}^{\lambda} \circ \mathbb{D}^{\lambda} =0$.  
We call a pair $(E, \mathbb{D}^{\lambda})$ with flat $\lambda$-connection $\mathbb{D}^{\lambda}$ a {\it $\lambda$-flat bundle}. We say it is {\it semisimple} if it is a direct sum of irreducible $\lambda$-flat bundles.  
\label{def:lc}
\end{definition}

\begin{remark} 
Note that the data in Definition \ref{def:lc}, while a priori a smooth $\lambda$-connection, will in fact be equivalent to a \emph{holomorphic} $\lambda$-connection in our development due to the assumption of compatibility with a Hermitian metric on $E$. 
\end{remark}

For $\lambda =1$ this object is a connection in usual sense; for $\lambda =0$, it becomes a Higgs field. 
Considered as a family over $\mathbb C$, this construction smoothly transforms a Higgs bundle with vanishing Chern classes into a flat bundle.  
More precisely, consider a Hermitian metric $h$ on $E$ with flat $\lambda$-connection $\mathbb{D}^{\lambda}$ as above. 
Then the $\lambda$-connection uniquely decomposes into  
$ \lambda \partial_{E, h} + \theta_{h} + \bar{\partial}_{E, h} + \lambda \theta_{h}^{*, h}$, 
where $\theta_{h}$ is a section of $\text{End} (E) \otimes \Lambda^{1,0}_{X}$ and $\theta_{h}^{*,h}$ is that of $\text{End} (E) \otimes \Lambda^{0, 1}_{X}$ (see e.g. \cite[\S 2]{Moch2}). 
Thus, $\mathbb{D}^{\lambda}$ with $\lambda =1$ is exactly the $\mathbb{D}_h^1$ before, and so does the $\lambda = 0$ case. 
Moreover, one can naturally form a family of moduli spaces on $\mathbb{P}^1$, where the central fibre is the moduli space of polystable Higgs bundles with vanishing Chern classes and the other fibres are those of semisimple $\lambda$-flat bundles.
This is known usually as the \emph{Deligne--Hitchin moduli space}.

We now state the classical nonabelian Hodge correspondence by Simpson \cite{Simp2, Simp4, Simp6, Simp7} in the following form:

\begin{theorem}
For each $\lambda \in \C$, there is a real analytic isomorphism (away from any singular loci) between the moduli spaces of polystable Higgs bundles with vanishing Chern classes and semisimple $\lambda$-flat bundles on a smooth, compact K\"{a}hler manifold.
\end{theorem}

The correspondence has been generalised to non-compact cases in a series of works, including those of Simpson \cite{Simp3} for the tame case and of Biquard--Boalch \cite{BiBo} for the wild case, both in complex dimension one. Mochizuki has extended both the tame and wild cases to arbitrary dimension \cite{Moch1, Moch2, Moch3}, while nonabelian Hodge theory has been developed in a number of further directions (e.g. \cite{FaRa} for cuspidal curves in a modular context).  We refer to a recent survey \cite{Huan} and its various references for the state of these developments.

\section{The equations on compact K\"{a}hler surfaces}
\label{Sec:4}

In this section, we describe the equations \eqref{eq:tzKW} (the Kapustin--Witten equations for $t=0$) and \eqref{eq:tnzKW}  (the Kapustin--Witten equations for $t \neq 0$) on a compact K\"{a}hler surface.
It turns out that they become familiar equations in nonabelian Hodge theory described in Section \ref{Sec:3}.

Let $X$ be a compact K\"{a}hler surface $X$ with K\"{a}hler form $\omega$, and let $E$ be a Hermitian vector bundle of rank $r$ on $X$ with Hermitian metric $k$.  
We assume $c_1(E) =0$ throughout this section. 
We denote by $\mathcal{A}_{(E,k)}$ the space of $k$-unitary connections on $E$ and by $\text{End} (E)_0 = \text{End} (E,k)_0$ the bundle of trace-free skew-Hermitian endomorphisms of $E$ on $X$. 
We decompose $T ^* X \otimes \C$ into $\Lambda^{1,0}_{X} \oplus \Lambda^{0,1}_{X}$ as well as the covariant derivative $d_A$ associated to $A \in \mathcal{A}_{(E,k)}$ into $\partial_{A} + \bar{\partial}_{A}$ by the complex structure of $X$; and write $\mathfrak{a} = \phi - \phi^*$, where $\phi \in \Gamma ( \text{End} (E)_0 \otimes \Lambda_{X}^{1,0} )$, and $\phi^{*} \in \Gamma ( \text{End} (E)_0 \otimes \Lambda_{X}^{0,1} )$ is the adjoint of $\phi$ with respect to $k$. 
Note that $\Lambda^{1,0}_X$ is the same as the holomorphic cotangent bundle $\Omega^1_X$ of $X$.

\subsection{The equations for $t=0$}

For $t =0$, Nakajima \cite[\S 6(iii)]{Naka} and the third-named author \cite[Prop.~2.5]{Tana} obtained the following:

\begin{proposition}
Let $E$ and $X$ be as above. Then, the Kapustin--Witten equations \eqref{eq:tzKW} for $t=0$ has the following form on a smooth, compact K\"{a}hler surface, searching for $(A, \phi) \in \mathcal{A}_{(E,k)} \times \Gamma ( \emph{End} (E)_0 \otimes \Lambda_{X}^{1,0} )$: 
\begin{gather}
F_{A}^{0,2}= 0, \, \dbar_{A} \phi  =0 , \,  \phi \wedge \phi =0, 
\label{eq:tzKWJ1} \\
\Lambda ( F_{A}^{1,1} + [ \phi \wedge \phi^{*} ] ) = 0,  
\label{eq:tzKWJ2}  
\end{gather}
where $\Lambda := ( \wedge \omega)^{*}$. 
\end{proposition}

We then reformulate the system of the above equations \eqref{eq:tzKWJ1},  \eqref{eq:tzKWJ2} in terms of Hermitian metrics on $E$  instead of connections of it to clarify a link to nonabelian Hodge Theory described in Section \ref{Sec:3}.

\begin{proposition}
Let $E$ be a complex vector bundle of rank $r$ with $c_1(E) =0$ over a smooth, compact K\"{a}hler surface $X$. 
Fix a Hermitian metric $k$ on $E$.  Then, searching for a solution $(A, \phi) \in \mathcal{A}_{(E,k)} \times \Gamma ( \emph{End} (E)_0 \otimes \Lambda_{X}^{1,0} )$ to the system of the equations \eqref{eq:tzKWJ1},  \eqref{eq:tzKWJ2} on $E$ is equivalent to looking for a Hermitian-Yang--Mills metric on a Higgs bundle $(E, \bar{\partial}_{E}, \theta)$ with $\emph{tr} \, \theta  =0$ on $X$.  
This correspondence is one-to-one up to gauge transformations. 
\label{Prop:mhym}
\end{proposition}

\begin{proof}
The proof here is similar to that by Bradlow for the vortex equations in \cite[\S 3]{Brad} and uses some results therein (see also \cite[Chap.VI]{Koba}). 
First, if $(A, \phi)  \in \mathcal{A}_{(E,k)} \times \Gamma ( \text{End} (E)_0 \otimes \Lambda_{X}^{1,0} )$ is a solution to the equation \eqref{eq:tzKWJ1}, 
we are given a connection $A$ on $E$ with $F_{A}^{0,2}=0$, namely, we have a holomorphic structure $\bar{\partial}_{A}$ defined by $A$ on $E$. 
In addition, we have $\phi \in \Gamma ( \text{End} (E)_0 \otimes \Lambda_{X}^{1,0} )$ with $\bar{\partial}_{A} \phi =0$ and $\phi \wedge \phi =0$. 
Therefore, $(E, \bar{\partial}_{A}, \phi)$ defines a Higgs bundle with $\text{tr} \, \phi =0$ on $X$. 
Then, the equation \eqref{eq:tzKWJ2} implies that $k$ gives a unique Hermtian-Yang--Mills metric on the Higgs bundle $(E, \bar{\partial}_{A}, \phi)$, since $A$ is the unique $k$-unitary connection compatible with the holomorphic structure $\bar{\partial}_{A}$ in this case.

Conversely, suppose that we are given a Higgs bundle $(E, \bar{\partial}_{E}, \theta)$ with holomorphic structure $\bar{\partial}_{E}$ and Higgs field $\theta$ with $\text{tr} \, \theta =0$. 
We consider a Hermitian metric $h$ on $E$, which is not necessarily the same as $k$ that we fixed at the beginning. 
Then, there exists a positive and self-adjoint endomorphism of $E$ with respect to $k$, we denote it by $H$, satisfying $h = k H$. This $H$ decomposes uniquely as $H = g^* g$, where $g$ is a complex gauge transformation. 
Define a holomorphic structure $\bar{\partial}_{E,k} := g ( \bar{\partial}_{E} ) =  g  \bar{\partial}_{E} g^{-1}$ on $E$ and a section $\phi := g (\theta ) = g  \theta g^{-1} \in  \Gamma ( \text{End} (E)_0 \otimes \Lambda_{X}^{1,0} )$ by the gauge transformation. 
Then, the tuple  $(E, \bar{\partial}_{E, k}, \phi)$ is a Higgs bundle on $X$, namely, $\bar{\partial}_{E,k} \phi =0, \,  \phi \wedge \phi =0$. 
We now consider a unique $k$-unitary connection $A := D^1_{\bar{\partial}_{E,k},k} = \bar{\partial}_{E,k} + \partial_{E, k}$ compatible with $\bar{\partial}_{E,k}$ and with $(1,0)$-connection $\partial_{E, k}$.   
We obviously have $F_A^{0,2}=0$. 
It is also straightforward to see that $ F_A^{1,1} = g  F_{h}^{1,1}  g^{-1}$, where $F_h^{1,1}$ is the $(1,1)$-part of the curvature $F_h$ of the $h$-unitary connection $D_{\bar{\partial}_{E}, h}^1$ compatible with $\bar{\partial}_{E}$, and $[ \phi \wedge \phi^* ] = g  [ \theta \wedge \theta^{*,h} ] g^{-1}$ (cf. \cite[Lem.3.4]{Brad}), namely, we have 
$\Lambda ( F_A^{1,1}  + [ \phi \wedge \phi^* ] )= g ( \Lambda ( F_{h}^{1,1} + [ \theta \wedge \theta^{*,h} ]) ) g^{-1} $.
Hence, $(A, \phi)$ satisfies the equations \eqref{eq:tzKWJ1},  \eqref{eq:tzKWJ2} if and only if $h$ is a Hermitian-Yang--Mills metric on the Higgs bundle $(E, \bar{\partial}_{E}, \theta)$.  
This construction obviously gives the inverse to the other one at the top of the proof up to gauge transformations. 
\end{proof}

Note that  Simpson \cite{Simp1} proves that a Hermitian-Yang--Mills metric on a Higgs bundle exists if and only if the Higgs bundle is polystable in the sense of Definition \ref{Def:stable} as mentioned also in Section \ref{Sec:3.1}.

\subsection{The equations for $t \neq 0$}

We first remark that it is straightforward to see that Theorem \ref{thm:GU} holds for $G=U(r)$  with $c_1 (P) = c_2 (P) =0$. 
Then, for $t \in \R \setminus \{ 0 \}$, we obtain the following:

\begin{proposition}
Let $E$ be a Hermitian vector bundle of rank $r$ over a smooth, compact K\"{a}hler surface $X$. 
Suppose that $c_1 (E) = c_2 (E) = 0$. Then, the Kapustin--Witten equations \eqref{eq:tnzKW} for $t \neq 0$ on $X$ reduces to the following equations for $(A, \phi) \in \mathcal{A}_{(E,k)} \times \Gamma ( \emph{End} \, (E)_0 \otimes \Lambda_{X}^{1,0} )$: 
\begin{gather}
F_{A}^{0,2}= 0, \, \bar{\partial}_{A} \phi  = 0, 
\, \phi \wedge \phi =0,
\label{eq:tnzKWJ1}  \\
F_{A}^{1,1} + [ \phi \wedge \phi^{*} ] = 0 , \,   \partial_A \phi = 0.
\label{eq:tnzKWJ2} 
\end{gather}
\label{prop:tnzKW}
\end{proposition}

\begin{proof}
We utilise a method by Nakajima \cite[\S 6(iii)]{Naka}. 
Recall first that the following Bochner--Weitzenb\"ock formula for $\mathfrak{a} \in \Gamma (\text{ad} (P) \otimes \Lambda^1_{X})$: 
\begin{equation}
d_A d_A^* \mathfrak{a} + d_A^* d_A \mathfrak{a} = \nabla_A^* \nabla_A \mathfrak{a}  + \text{Ric} \circ  \mathfrak{a} + [F_A,\mathfrak{a}], 
\label{eq:ci}
\end{equation}
where $\text{Ric}$ denotes the Ricci transformation of the underlying Riemannian metric of $X$  
(see e.g. Bourguignon--Lawson \cite[Theorem~3.2]{BoLa} for more details). 
By taking the $L^2$-inner product of the identity \eqref{eq:ci} with $\mathfrak{a}$, we obtain 
\begin{equation}
\begin{split}
\int_X | d_A^* \mathfrak{a} |^2 &+\int_X | d_A \mathfrak{a} |^2 + \frac{1}{2} \int_X | F_A - [ \mathfrak{a} \wedge \mathfrak{a} ]|^2 \\
&= \int_X | \nabla_A \mathfrak{a} |^2 + \frac{1}{2} \int_X | F_A |^2 + \frac{1}{2} \int_X | [ \mathfrak{a} \wedge \mathfrak{a} ] |^2+ \int_X s| \mathfrak{a} |^2, 
\end{split}
\label{eq:BW}
\end{equation}
where $s$ is the scalar curvature of the underlying Riemannian metric on $X$.

Abusing the notation, we denote $ i \phi + i \phi^*$ by $i \mathfrak{a}$ on a compact K\"{a}hler surface $X$. 
Observe that the left hand side of \eqref{eq:BW} is zero, if $(A, \mathfrak{a})$ is a solution to \eqref{eq:tnzKW}. 
On the other hand, the right hand side of \eqref{eq:BW} is unchanged by replacing $\mathfrak{a}$ by $i \mathfrak{a}$. 
Hence, if $(A, \mathfrak{a})$ is a solution to \eqref{eq:tnzKW}, so does $(A, i \mathfrak{a})$.  
In other words, if $(A, \phi)$ is a solution to \eqref{eq:tnzKW}, so does $(A, i \phi)$.

The $(2,0)$-part of the first equation in \eqref{eq:tnzKW}, namely, $F_{A} - [ \mathfrak{a} , \mathfrak{a} ]=0$, reads $F_A^{2,0} -  \phi  \wedge \phi =0$.  Since $(A, \phi)$ and $(A, i \phi )$ are solutions to the equation simultaneously, we obtain $F_A^{2,0} =0 $ and $\phi \wedge \phi =0$. 
From the same argument for the $(0,2)$-part of the first equation in \eqref{eq:tnzKW}, we obtain $F_A^{0,2} =0$ and $\phi^*  \wedge \phi^* =0$. 
We obviously obtain $F_{A}^{1,1} + [ \phi \wedge \phi^{*} ] = 0$ from the $(1,1)$-part of the first equation in \eqref{eq:tnzKW}.

In addition, from the $(2,0)$ and $(0,2)$-parts of the equation $d_A \mathfrak{a} =0$, we obtain $\partial_A \phi = \bar{\partial}_{A} \phi^* =0$. 
Since $d_A (  i \mathfrak{a} ) =0$, we obtain $ \bar{\partial}_A \phi + \partial \phi^* =0$.
From this with the $(1,1)$ part of  the original $d_A \mathfrak{a} =0$, namely,  $ \bar{\partial}_A \phi - \partial \phi^* =0$, we obtain $ \bar{\partial}_A \phi = 0$.

For the opposite direction, since $d_{A}^{*} \mathfrak{a} = \Lambda (\bar{\partial}_{A} \phi + \partial_{A} \phi^{*})$ via the K\"{a}hler identities (see \cite[\S 2.2]{Tana}), it is automatic to obtain $d_{A}^{*} \mathfrak{a} =0$. 
\end{proof}

\begin{remark}
When the scalar curvature $s$ of the underlying manifold $X$ is positive, the identity \eqref{eq:BW} can be used to provide a {\it vanishing theorem} for $\mathfrak{a}$ in the $t \neq 0$ Kapustin--Witten equations \eqref{eq:tnzKW}.
\end{remark}

\begin{remark}
Ward \cite{Ward} discusses the system of the equations \eqref{eq:tnzKWJ1}, \eqref{eq:tnzKWJ2} as an instance of higher-dimensional completely integrable systems, generalising the notion from real dimension $2$ that was initiated by Hitchin \cite{Hitc2}. The setting for Ward's construction is real dimension $2k$ with explicit examples in real dimension $4$. 
\end{remark}

We then reformulate the system of the above equations \eqref{eq:tnzKWJ1}, \eqref{eq:tnzKWJ2} in terms of Hermitian metrics on $E$ as before.

\begin{proposition}
Let $E$ be a complex vector bundle of rank $r$ with over a smooth, compact K\"{a}hler surface $X$. Suppose $c_1 (E) = c_2(E) =0$.  Fix a Hermitian metric $k$ on $E$. 
Then, searching for a solution $(A, \phi) \in \mathcal{A}_{(E,k)} \times \Gamma ( \emph{End} (E)_0 \otimes \Lambda_{X}^{1,0} )$ to the system of the equations \eqref{eq:tnzKWJ1},  \eqref{eq:tnzKWJ2} on $E$ is equivalent to looking for a pluri-harmonic metric on a Higgs bundle $(E, \bar{\partial}_{E}, \theta)$ with $\emph{tr} \, \theta =0$ on $X$.  
This correspondence is one-to-one up to gauge transformations. 
\end{proposition}

\begin{proof}
The proof here goes along the same line with that of Proposition \ref{Prop:mhym}. 
If $(A, \phi) \in \mathcal{A}_{(E,k)} \times \Gamma ( \text{End} (E)_0 \otimes \Lambda_{X}^{1,0} )$ is a solution to the equations \eqref{eq:tnzKWJ1}, we are given a Higgs bundle $(E, \bar{\partial}_{A}, \phi)$ with $\text{tr} \, \phi = 0$ as before. 
Then, \eqref{eq:tnzKWJ2} claims that $k$ is a pluri-harmonic metric on the Higgs bundle $(E, \bar{\partial}_{A}, \phi)$.

Conversely, start with a Higgs bundle $(E, \bar{\partial}_{E}, \theta)$ with holomorphic structure $\bar{\partial}_{E}$ and Higgs field $\theta$. We consider a Hermitian metric $h$ on $E$, which need not be the same as the fixed Hermitian metric $k$ at the beginning. 
Then, we have a positive and self-adjoint endomorphism $H$ of $E$ with respect to $k$, which satisfies $h = k H$. We decompose this $H$ as $H = g^* g$, where $g$ is a complex gauge transformation. 
Define a holomorphic structure $\bar{\partial}_{E,k} := g ( \bar{\partial}_{E} ) =  g \bar{\partial}_{E} g^{-1}$ on $E$  and a section $\phi := g  (\theta ) = g   \theta g^{-1} \in \Gamma ( \text{End} (E)_0 \otimes \Lambda_{X}^{1,0} )$ by the gauge transformation. 
Then, the tuple  $(E, \bar{\partial}_{E, k}, \phi)$ is a Higgs bundle on $X$, that is, $\bar{\partial}_{E,k} \phi =0, \, \phi \wedge \phi =0$. 
Let us then consider a unique $k$-unitary connection $A := D^1_{\bar{\partial}_{E,k},k}  =  \bar{\partial}_{E,k} + \partial_{E, k}$ compatible with $\bar{\partial}_{E,k}$ and with $(1,0)$-connection $\partial_{E, k}$ as before.

On the other hand,  we have a  unique $h$-unitary connection $D_{\bar{\partial}_E, h}^1$ compatible with $\bar{\partial}_{E}$ and with $(1,0)$-connection $\partial_{E, h}$. 
Let us then consider $\mathbb{D}_{h}^1 := D_{\bar{\partial}_{E}, h}^1 + \theta + \theta^{*,h}$.  
This $\mathbb{D}_{h}^1 $ is a connection which is compatible with the metric $h$. 
Since $(E, \bar{\partial}_{E}, \theta)$ is a Higgs bundle on $X$, we have $\mathbb{D}_h^1 \circ \mathbb{D}_{h}^1  =  F_{\bar{\partial}_{E}, h}^{1,1} + [\theta , \theta^{*,h} ] + \partial_{E, h} \theta + \bar{\partial}_{E} \theta^{*,h}$, where $F_{h}^{1,1}$ is the $(1,1)$-part of the curvature $F_{h}$ of $D_{\bar{\partial}_{E}, h}^1$.  
Then, it is just clear that $(A, \phi)$ satisfies the equations \eqref{eq:tnzKWJ1},  \eqref{eq:tnzKWJ2} if and only if $\mathbb{D}_{h}^{1} \circ \mathbb{D}_{h}^{1}$ vanishes, namely, $h$ is a pluri-harmonic metric on the Higgs bundle $(E, \bar{\partial}_{E}, \theta)$, since $F_A^{1,1} + [ \phi \wedge \phi^* ] = g ( F_{h}^{1,1}  +  [ \theta \wedge \theta^{*,h} ] ) g^{-1}$, $\partial_{E,k} \phi = g  ( \partial_{E,h} \theta ) g^{-1}$, and $\bar{\partial}_{E,k} \phi^{*,k} = g ( \bar{\partial}_{E} \theta^{*,h} ) g^{-1}$ (cf. \cite[Lem.3.4]{Brad}).  
This construction certainly gives the inverse to the other up to gauge transformations. 
\end{proof}

A pluri-harmonic metric on a Higgs bundle is a Hermitian metric which determines the flat $\lambda$-connection at $\lambda = 1$,  as described in Section \ref{Sec:3}. 
Hence, it is natural to think of the moduli space of solutions to the $t \neq 0$ Kapustin--Witten equations \eqref{eq:tnzKWJ1},  \eqref{eq:tnzKWJ2} on a compact K\"{a}hler surface as that of semisimple flat bundles. Furthermore, by the nonabelian Hodge correspondence described in Section \ref{Sec:3}, we obtain the following:

\begin{theorem}
Let $X$ be a smooth, compact K\"{a}hler surface $X$, and let $E$ be a Hermitian vector bundle on $X$ with $c_1(E)=c_2 (E) =0$. 
Then the moduli space of solutions to the $t=0$ Kapustin--Witten equations \eqref{eq:tzKWJ1}, \eqref{eq:tzKWJ2} on $E$ over $X$ is real-analytically isomorphic (away from any singular loci) to the moduli space of solutions to the $t \neq 0$ Kapustin--Witten equations \eqref{eq:tnzKWJ1}, \eqref{eq:tnzKWJ2} on $E$ over $X$. 
\label{th:4}
\end{theorem}

This observation may motivate the study of the $P=W$ phenomenon developed for the nonabelian Hodge theory of curves by de Cataldo--Hausel--Migliorini \cite{CaHaMi} and others to the complex projective surface case. The $P=W$ conjecture posits that two different filtrations on the cohomologies of the moduli spaces coincide --- the theorem above may connect the phenomenon to the the topology of moduli spaces in the complex projective surface case and furthermore to their interpretation in terms of super Yang--Mills theory in physics.  The relationship between opers and nonabelian Hodge-theoretic deformations, as per Dumitrescu--Fredrickson--Kydonakis--Mazzeo--Mulase--Neitzke \cite{DFKMMN}, might also be discussed now in a higher-dimensional context.

\section{The expected dimensions of the moduli spaces on closed four-manifolds}
\label{Sec:5}

In the previous section, we saw the moduli spaces of solutions to the Kapustin--Witten equations \eqref{eq:pKW1}--\eqref{eq:pKW3} at different values of $t$  are real-analytically isomorphic away from singular loci on a smooth, compact K\"{a}hler surface. 
A natural question to ask then is that to what extent of generality this might hold.  
In this section, we prove that the expected dimensions of the moduli spaces of solutions to the $t =0$ Kapustin--Witten equations \eqref{eq:tzKW} and $t \neq 0$ Kapustin--Witten equations \eqref{eq:tnzKW} coincide on a smooth, closed four-manifold.

The expected dimension is the dimension of the moduli space when it is smooth and unobstructed. 
This can be computed by considering Atiyah--Hitchin--Singer type deformation complex, which is formed from the system of linearised equations in each problem, and they are elliptic if the problem is well-posed. 
Then, the expected dimension is given by $-1$ times the index of the complex, since the Zariski tangent space of the moduli space is identified with the first cohomology of the complex.

To set the stage for this, let $X$ be a closed, oriented, smooth four-manifold, and let $P \to X$ be a principal $G$-bundle over $X$, where $G$ is a compact Lie group. We write $\Lambda^p_X := \Lambda^p T^* X$, for $p=0, \dots ,4$. 
We denote by $\Lambda_X^{\pm}$ the self-dual part and the anti-self-dual part of $\Lambda_X^2$ respectively, and by $\pi^{\pm}: \Gamma (\text{ad} (P) \otimes \Lambda^2) \to \Gamma (\text{ad} (P) \otimes \Lambda^{\pm})$ the projections. We define $d_{A}^{\pm} := \pi^{\pm} \circ d_A$.

\subsection{The $t=0$ case}

For $t =0$, the linearisation of the system of the equation \eqref{eq:tzKW} with gauge fixing fits into to the following Atiyah--Hitchin--Singer type complex: 
\begin{equation}
\begin{split}
0 \longrightarrow \Gamma ( \text{ad} (P) & \otimes \Lambda^0_X )  \xrightarrow{(d_A , [ \mathfrak{a} , \, \, \, ])}  \Gamma ( \text{ad} (P) \otimes ( \Lambda^1_X  \oplus \Lambda^1_X ) ) \\ 
& \xrightarrow{D_{A, t=0}}  \Gamma ( \text{ad} (P) \otimes \Lambda^{+}_X )  \times 
\Gamma ( \text{ad} (P) \otimes ( \Lambda^{-}_X \oplus \Lambda^0_X ) ) \longrightarrow 0,  \\
\end{split} 
\label{eq:AHSt0}
\end{equation}
where $D_{A, t =0} (a , \alpha ) := (d_A^{+} a +  [ \mathfrak{a} \wedge \alpha]^{+} + [ \alpha  \wedge \mathfrak{a} ]^{+}, ( d_{A}^{-} \alpha , d_{A}^{*} \alpha ) )$ for $(a, \alpha ) \in \Gamma ( \text{ad} (P) \otimes ( \Lambda^1_X  \oplus \Lambda^1_X ) ) $.  This is a complex if $(A, \mathfrak{a})$ satisfies \eqref{eq:tzKW}, and it is elliptic.

Recall that we denote by $k(G)$ the integer, which is equal to $c_2 ( P \otimes \C) [X]$ if $G$ is $SU(2)$; and $- \frac{1}{4} p_1 (P) [X]$ if $G$ is $SO(3)$.   

\begin{theorem}
If the structure group $G$ of $P$ is one of $SU(2)$ or $SO(3)$, then the index of the elliptic complex \eqref{eq:AHSt0} is given by $ - 16 k (G) + 3 \chi (X)$, where $\chi (X) $ is the Euler number of $X$, namely, $\chi (X) = \sum_{i=0}^{4} (-1)^i b^{i} (X)$.   
\end{theorem}

\begin{proof}
We ignore the zero-th order  terms in \eqref{eq:AHSt0}, since the index only depends upon the symbol of the complex.  
Namely, we consider the following elliptic operator for the calculation of the index: 
\begin{equation}
 L_{A, t=0} :  \Gamma ( \text{ad} (P) \otimes  ( \Lambda^1_X  \oplus \Lambda^1_X ) )   \longrightarrow   \Gamma ( \text{ad} (P) \otimes ( \Lambda^{+}_X  \oplus \Lambda^{-}_X \oplus \Lambda^0_X \oplus  \Lambda^0_X ) )  , 
\label{eq:EOt0}
\end{equation}
which is defined as $L_{A, t=0} (A, \alpha ) := ( d_{A}^{+} a , d_{A}^{-} \alpha , d_{A}^{*} \alpha, d_A^{*} a)$ for $(A, \alpha) \in \Gamma ( \text{ad} (P) \otimes ( \Lambda^1_X  \oplus \Lambda^1_X ) )$.    

The operator \eqref{eq:EOt0} decomposes into the following two elliptic operators: 
\begin{gather}
  L^+ :   \Gamma ( \text{ad} (P)  \otimes \Lambda^1_X ) 
 \xrightarrow{(d_A^+ , d_A^*)}  \Gamma ( \text{ad} (P) \otimes ( \Lambda^+_X \oplus \Lambda^0_X )),  
 \label{eq:EOp} \\
   L^- :   \Gamma ( \text{ad} (P)  \otimes \Lambda^1_X ) 
 \xrightarrow{(d_A^-, d_A^* )}  \Gamma ( \text{ad} (P) \otimes ( \Lambda^-_X \oplus \Lambda^0_X )),  
 \label{eq:EOm} 
\end{gather}
These are familiar operators in gauge theory, and the indices are given by $8 k(G) - 3( 1 - b^1 + b^{+})$ and $8 k(G) - 3 ( 1 - b^1 + b^{-})$ respectively (see e.g. Atiyah--Hitchin--Singer \cite[\S 6]{AtHiSi},  Donaldson--Kronheimer \cite[\S 4.2]{DoKr}), where $b^+$ and $b^-$ is the dimension of maximal positive  and negative definite subspaces of the intersection form on $H_2(X , \Z)$ respectively. 
Since the elliptic operator \eqref{eq:EOt0} is the direct sum of the elliptic operators \eqref{eq:EOp} 
and \eqref{eq:EOm}, the index of \eqref{eq:EOt0} is given by $ 16 k(G) - 3 ( 2 - 2 b^1 + b^2)$. From the Poincar\'
e duality, this equals to $16 k(G) -  3( 1 -  b^1 + b^2 - b^3 + b^4)$,  thus the assertion holds.   
\end{proof}

\subsection{The $t \neq 0$ case}

For $t \in \R \setminus \{ 0 \}$, the linearisation of the system of the equation \eqref{eq:tnzKW} with gauge fixing fits into to the following Atiyah--Hitchin--Singer type complex: 
\begin{equation}
\begin{split}
0 \longrightarrow \Gamma ( \text{ad} (P) \otimes \Lambda^0_X )   \xrightarrow{(d_A , [ \mathfrak{a} , \, \, \, ])}  
 \Gamma ( & \text{ad} (P) \otimes ( \Lambda^1_X \oplus \Lambda^1_X) )  \\ 
& \xrightarrow{D_{A, t \neq 0}}  \Gamma ( \text{ad} (P) \otimes (\Lambda^2_X \oplus \Lambda^0_X ) )    \longrightarrow 0,  \\
\end{split} 
\label{eq:AHStn0}
\end{equation}
where $D_{A, t \neq 0} (a , \alpha ) := (d_A  a  + d_{A} \alpha +  [ \mathfrak{a} \wedge \alpha] + [ \alpha \wedge \mathfrak{a} ], d_{A}^{*} \alpha  - * [ \alpha \wedge * \mathfrak{a}] ) $ for $(a, \alpha ) \in \Gamma ( \text{ad} (P) \otimes ( \Lambda^1_X \oplus \Lambda^1_X) )$.  This is a complex if $(A, \mathfrak{a})$ satisfies \eqref{eq:tnzKW}, and it is elliptic.

The following was obtained by Mazzeo--Witten \cite[Propositon 4.1]{MaWi}, they did it also for the case of four-manifolds with boundary. We give a proof of it here for the reader's convenience.

\begin{theorem}[Mazzeo--Witten]
Consider a principal $G$-bundle over $X$, where $G$ is $SU(2)$ or $SO(3)$. 
Suppose that $k(G) =0$. Then the index of the elliptic complex \eqref{eq:AHStn0} is given by $3 \chi (X)$.
\end{theorem}

\begin{proof}
We ignore the zero-th order terms in \eqref{eq:AHStn0} here as well. 
A trick here is to identify the second summand $\Lambda^1_X$ in the second term of \eqref{eq:AHStn0} with $\Lambda^3_X$ and the second summand $\Lambda^0_X$ in the third term of \eqref{eq:AHStn0} with $\Lambda^4_X$ by the Hodge star operator on $X$. 
Then by holding them into the odd degree part and the even degree part, the index calculation of \eqref{eq:AHStn0} reduces to that of the following twisted de Rham operator: 
\begin{equation}
L_{A, t \neq 0} : \Gamma ( \text{ad} (P) \otimes ( \Lambda^1_X \oplus \Lambda^3_X) )   \\
\longrightarrow  \Gamma ( \text{ad} (P) \otimes ( \Lambda^0_X \oplus \Lambda^2_X \oplus \Lambda^4_X ))   . 
\label{eq:EOtdr}
\end{equation}
This operator is defined by $L_{A, t \neq 0} (a, \alpha' ) := ( d_A^{*}  a , d_A a + d_A^{*} \alpha' , d_A \alpha')$ for $ (a, \alpha' ) \in \Gamma ( \text{ad} (P) \otimes ( \Lambda^1_X \oplus \Lambda^3_X) )$. 
The topological index of the elliptic operator \eqref{eq:EOtdr} is given by 
$$ 
- \frac{\text{ch} ( \text{ad} (P) \otimes \C) \, \text{ch} (\sum_{i=0}^{4} (-1)^i \Lambda^i_{TX \otimes \C}) \, \text{td} (TX \otimes \C)}{e(TX)} [X], 
$$
where $\text{ad} ( P ) \otimes \C := P \times_{ad} (\mathfrak{g} \otimes \C)$ 
(See e.g. Shanahan \cite[I]{Shan}  for more details).     
From the assumption, we have $\text{ch} ( \text{ad} (P) \otimes \C) = \dim \mathfrak{g} =3$. 
In addition, the remaining factor 
$ - \frac{\text{ch} (\sum_{i=0}^{4} (-1)^i \Lambda^i_{TX \otimes \C}) \, \text{td} (TX \otimes \C)}{e(TX)} [X]$   
is $-1$ times the toplogical index for the de Rham complex on $X$ (see \cite[II.4]{Shan}), hence it equals $-1$ times the Euler number of $X$. 
Thus the assertion holds. 
\end{proof}

We immediately obtain the following:

\begin{Corollary}
Consider a principal $G$-bundle over a closed, oriented, smooth four-manifold $X$, where $G$ is one of $SU(2)$ or $SO(3)$.  Assume that $k(G)=0$. 
Then the indices of \eqref{eq:AHSt0} and \eqref{eq:AHStn0} coincide. 
In other words, so do the expected dimensions of the moduli spaces of solutions to the Kapustin--Witten equations for $t=0$ and $t \neq0$ over $X$. 
\label{cor:vd}
\end{Corollary}

\section*{Acknowledgements} C.-C. L. was partially supported by Ministry of Science and Technology of Taiwan under grant number 109-2115-M-006-011. S. R. was partially supported by an NSERC Discovery Grant. 
Y. T. was partially supported by the Simons Collaboration on Special Holonomy in Geometry, Analysis and Physics and JSPS Grant-in-Aid for Scientific Research number JP16K05125 during the preparation of this manuscript.  The authors thank Christopher Beem, Laura Fredrickson, Sergei Gukov, Hiroshi Iritani, Rafe Mazzeo, Takuro Mochizuki, \'Akos Nagy, Hiraku Nakajima, and Sakura Sch\"afer-Nameki for helpful conversations.  S. R. is grateful to Andrew Dancer and Frances Kirwan for their hospitality and enlightening discussions during a November 2018 visit to the Oxford Mathematical Institute, where the second- and third-named authors initiated this project. Y. T. is grateful to Hiraku Nakajima for the support and hospitality at Kavli IPMU in Autumn 2020. S. R. and Y. T. thank the Mathematisches Forschungsinstitut Oberwolfach (MFO) and the organizers of the May 2019 Workshop on Geometry and Physics of Higgs Bundles (Lara Anderson, Tam\'as Hausel, Rafe Mazzeo, and Laura Schaposnik) for a stimulating environment in which some formative ideas related to this project were discussed. 

%%%%%%%%%%%%%%%%%%%%%%%%%%%%%%%%%%%%%%%%%%%%%%%%%%%%%%%%%%%%%%%%%%%%%%%%%%%%%%%%%%%%%%%

\medskip

%%%%%%%%%%%%%%%%%%%%%%%%%%%%%%%%%%%%%%%%%%%%%%%%%%%%%%%%%%%%%%%%%%%%%%%%%%%%%%%%%%%%%%%%%%%%%%%%%%%%%%


\begin{thebibliography}{99}


\bibitem{AtHiSi} M. F. Atiyah, N. J. Hitchin, and I. M. Singer, {\it Self-duality in four-dimensional Riemannian geometry}, Proc. Roy. Soc. London Ser. A  362  (1978),  425--461. 

\bibitem{BaSc1} D. Baraglia and L. P. Schaposnik, {\it Higgs bundles and $(A,B,A)$-branes}, Comm. Math. Phys. 331 (2014), 1271--1300. \href{https://arxiv.org/abs/1305.4638}{arXiv:1305.4638}. 

\bibitem{BaSc2} D. Baraglia and L. P Schaposnik, {\it Real structures on moduli spaces of Higgs bundles}, Adv. Theor. Math. Phys. 20 (2016), 525--551. \href{https://arxiv.org/abs/1309.1195}{arXiv:1309.1195}. 

\bibitem{BiBo} O. Biquard and P. Boalch, {\it Wild non-abelian Hodge theory on curves}, Compos. Math. 140 (2004), 179--204.  \href{https://arxiv.org/abs/math/0111098}{arXiv:math/0111098}.

\bibitem{BGPH} I. Biswas, O. Garc\'{i}a-Prada, and J. Hurtubise, {\it Higgs bundles, branes and Langlands duality}, Comm. Math. Phys. 365 (2019), 
1005--1018. \href{https://arxiv.org/abs/1707.00392}{arXiv:1707.00392}. 

\bibitem{BoLa}
J-P. Bourguignon and H. B. Lawson, Jr., {\it Stability and isolation phenomena for Yang--Mills fields}, Comm. Math. Phys. 79 (1981), 189--230. 

\bibitem{Brad} S. B. Bradlow, {\it Vortices in holomorphic line bundles over closed K\"{a}hler manifolds},  
Comm. Math. Phys. 135 (1990), 1--17. 

\bibitem{CaHaMi}M. A. A. de Cataldo, T. Hausel, and L Migliorini, {\it Topology of Hitchin systems and Hodge theory of character varieties: the case $A_1$}, Ann. of Math. 175 (2012), 1329--1407.  
\href{https://arxiv.org/abs/1004.1420}{arXiv:1004.1420}. 
 
\bibitem{ChNg} T. H. Chen and B. C. Ng\^o, {\it On the Hitchin morphism for higher-dimensional varieties}, Duke Math. J.  169  (2020), 1971--2004. 
\href{https://arxiv.org/abs/1905.04741}{arXiv:1905.04741}. 
  
\bibitem{Corl}K. Corlette, {\it Flat $G$-bundles with canonical metrics},  
J. Differential Geom. 28 (1988), 361--382. 

\bibitem{DoPa} R. Donagi and T. Pantev, {\it Langlands duality for Hitchin systems}, Invent. Math. 189 (2012), 653--735. \href{https://arxiv.org/abs/math/0604617}{arXiv:math/0604617}. 

\bibitem{Dona}S. K. Donaldson, {\it Twisted harmonic maps and the self-duality equations},  
Proc. London Math. Soc. 55 (1987), 127--131. 

\bibitem{DoKr} S. K. Donaldson and P.B. Kronheimer, {\it The geometry of four-manifolds}, Oxford Mathematical Monographs, Oxford University Press, 1990. 

\bibitem{DFKMMN} O. Dumitrescu, L. Fredrickson, G. Kydonakis, R. Mazzeo, M. Mulase, and A. Neitzke, {\it Opers versus nonabelian Hodge}, \href{https://arxiv.org/abs/1607.02172}{arXiv:1607.02172}, 2016. 

\bibitem{EeSa} J. Eells, Jr. and J. H. Sampson, {\it Harmonic mappings of Riemannian manifolds}, Am. J. Math. 86 (1964) 109--160. 
 
\bibitem{ElPe}C. Elliott and V. Pestun, {\it Multiplicative Hitchin systems and supersymmetric gauge theory}, Selecta Math. (N.S.) 25 (2019), Paper No. 64, 82 pp. \href{https://arxiv.org/abs/1812.05516}{arXiv:1812.05516}. 

\bibitem{FaRa} C. Franc and S. Rayan, {\it Nonabelian Hodge theory and vector valued modular forms}, 95--118, Contemp. Math. 753, Amer. Math. Soc., Providence, RI, 2020. \href{https://arxiv.org/abs/1812.06180}{arXiv:1812.06180}. 

\bibitem{FGOP} E. Franco, P. Gothen, A. Oliveira, and A. Pe\'on-Nieto, {\it Unramified covers and branes on the Hitchin system}.  \href{https://arxiv.org/abs/1802.05237}{arXiv:1802.05237}. 2018. 

\bibitem{FJM} E. Franco, M. Jardim, and S. Marchesi, {\it Branes in the moduli space of framed sheaves},  Bull. Sci. Math. 141 (2017), 353--383. 
\href{https://arxiv.org/abs/1504.05883}{arXiv:1504.05883}. 

\bibitem{GaUh} M. Gagliardo and K. Uhlenbeck, 
{\it Geometric aspects of the Kapustin--Witten
equations}, 
J. Fixed Point Theory Appl. 11 (2012), 185--198. \href{https://arxiv.org/abs/1401.7366}{arXiv:1401.7366}. 

\bibitem{GaWi}
D. Gaiotto and E. Witten, {\it Knot invariants from four-dimensional gauge theory}, Adv. Theor. Math. Phys. 16 (2012),  935--1086. \href{https://arxiv.org/abs/1106.4789}{arXiv:1106.4789.} 

\bibitem{GaRa} A. Garc\'{i}a-Raboso and S. Rayan, {\it Introduction to nonabelian Hodge theory: flat connections, Higgs bundles and complex variations of Hodge structure}, ``Calabi--Yau varieties: arithmetic, geometry and physics'', Fields Institute Monograph 34 (2015), 131--171. \href{https://arxiv.org/abs/1406.1693}{arXiv:1406.1693}. 

\bibitem{GuWi} S. Gukov and E. Witten, {\it Gauge theory, ramification, and the geometric Langlands program}, Current developments in mathematics, 2006, 35--180, Int. Press, Somerville, MA, 2008.  \href{https://arxiv.org/abs/hep-th/0612073}{arXiv:hep-th/0612073}. 

\bibitem{HaTh} T. Hausel and M. Thaddeus, {\it Mirror symmetry, Langlands duality, and the Hitchin system}, Invent. Math. 153 (2003), 197--229.
\href{https://arxiv.org/abs/math/0205236}{arXiv:math/0205236}. 

\bibitem{Hitc} N. J. Hitchin, {\it The self-duality equations on a Riemann surface}, Proc. London Math. Soc. 55 (1987), 59--126. 

\bibitem{Hitc2} N. Hitchin, {\it Stable bundles and integrable systems}. Duke Math. J.  54 (1987), 91--114. 
 
\bibitem{Huan} P. Huang, {\it Non-Abelian Hodge theory and related topics}, SIGMA Symmetry Integrability Geom. Methods Appl. 16 (2020), 029, 34 pp. \href{https://arxiv.org/abs/1908.08348}{arXiv:1908.08348}.

\bibitem{JoYa} J. Jost and S-T. Yau, {\it Harmonic maps and group representations}, Differential geometry, 241--259, Pitman Monogr. Surveys Pure Appl. Math., 52, Longman Sci. Tech., Harlow, 1991.

\bibitem{Joyc1} D. Joyce, {\it An introduction to d-manifolds and
derived differential geometry}, pages 230--281 in L. Brambila-Paz et al., editors, Moduli spaces, L.M.S. Lecture Notes 411, Cambridge University Press, 2014. \href{https://arxiv.org/abs/1206.4207}{arXiv:1206.4207}.

\bibitem{Joyc2} D. Joyce, {\it D-manifolds and d-orbifolds: a
theory of derived differential geometry}, Preliminary version (2012) available at \url{https://people.maths.ox.ac.uk/~joyce/dmanifolds.html}.

\bibitem{JTU} D. Joyce, Y. Tanaka, and M. Upmeier, {\it  On orientations for gauge-theoretic moduli spaces}, Adv. Math. 362 (2020), 106957. \href{https://arxiv.org/abs/1811.01096}{arXiv:1811.01096}.

\bibitem{KaWi} A. Kapustin and E. Witten, {\it Electric-magnetic duality and the geometric Langlands
program}, Commun. Number Theory Phys. 1 (2007), 1--236. \href{https://arxiv.org/abs/hep-th/0604151}{hep-th/0604151}. 

\bibitem{Koba} S. Kobayashi, {\it Differential geometry of complex vector bundles}, Publications of the Mathematical Society of Japan, 15. Kan\^o Memorial Lectures, 5. Princeton University Press, Princeton, NJ (1987). 

\bibitem{Labo} F. Labourie, {\it Existence d'applications harmoniques tordues \`a valeurs dans les vari\'et\'es \`a courbure n\'egative}, Proc. Amer. Math. Soc. 111 (1991), 877--882.
 
\bibitem{MaWi} R. Mazzeo and E. Witten, {\it The Nahm pole boundary condition}, in ``The influence of Solomon Lefschetz in geometry and topology'',  Contemp. Math., 621, Amer. Math. Soc., Providence, RI,  (2014), 171--226. \href{https://arxiv.org/abs/1311.3167}{arXiv:1311.3167}. 

\bibitem{Moch1} T. Mochizuki, {\it Kobayashi--Hitchin correspondence for tame harmonic bundles and an application},  Ast\'erisque  309  (2006).  \href{https://arxiv.org/abs/math/0411300}{arXiv:math/0411300}. 
 
\bibitem{Moch2} T. Mochizuki, {\it Kobayashi--Hitchin correspondence for tame harmonic bundles. II}, Geom. Topol.  13  (2009), 359--455.  \href{https://arxiv.org/abs/math/0602266}{arXiv:math/0602266}. 

\bibitem{Moch3} T. Mochizuki, {\it Wild harmonic bundles and wild pure twistor $D$-modules}, Ast\'erisque  340 (2011). 
\href{https://arxiv.org/abs/0803.1344}{arXiv:0803.1344}. 

\bibitem{Naka} H. Nakajima, {\it Towards a mathematical definition of Coulomb branches of 3-dimensional $\mathcal{N}=4$  gauge theories, I.} Adv. Theor. Math. Phys.  20  (2016), 595--669. \href{https://arxiv.org/abs/1503.03676}{arXiv:1503.03676}.

\bibitem{PTVV} T. Pantev, B. To\"en, M. Vaqui\'e, and G. Vezzosi, {\it  Shifted symplectic structures}, Publ. Math. I.H.E.S. 117 (2013), 271--328. \href{http://arxiv.org/abs/1111.3209}{arXiv:1111.3209}.

\bibitem{RaSc} S. Rayan, and L. P. Schaposnik, {\it Moduli spaces of generalized hyperpolygons}. \href{https://arxiv.org/abs/2001.06911}{arXiv:2001.06911}. 2020. 

\bibitem{Shan} P. Shanahan, {\it The Atiyah--Singer index theorem. An introduction}, Lecture Notes in Math. 638, Springer, Berlin, 1978.

\bibitem{Simp1} C. T. Simpson, {\it Constructing variations of Hodge structure using Yang-Mills theory and applications to uniformization}, J. Amer. Math. Soc.  1  (1988),  867--918. 

\bibitem{Simp2} C. T. Simpson, {\it Nonabelian Hodge theory},  Proceedings of the International Congress of Mathematicians, Vol. I (Kyoto, 1990), 747--756. 

\bibitem{Simp3} C. T. Simpson, {\it Harmonic bundles on noncompact curves},  J. Amer. Math. Soc. 3 (1990), 713--770.

\bibitem{Simp4} C. T. Simpson, {\it Higgs bundles and local systems},  Inst. Hautes \'{E}tudes Sci. Publ. Math.  75  (1992), 5--95. 

\bibitem{Simp5} C. T. Simpson, {\it Moduli of representations of the fundamental group of a smooth projective variety I}, Inst. Hautes \'{E}tudes Sci. Publ. Math. No. 79, (1994), 47--129. 

\bibitem{Simp6} C. T. Simpson, {\it Moduli of representations of the fundamental group of a smooth projective variety II}, Inst. Hautes \'{E}tudes Sci. Publ. Math. No. 80, (1994), 5--79. 

\bibitem{Simp7} C. T. Simpson, {\it The Hodge filtration on nonabelian cohomology} in Algebraic Geometry: Santa Cruz 1995, Proc. Sympos. Pure Math. 62, Amer. Math. Soc. (1997),  217--281. \href{https://arxiv.org/abs/alg-geom/9604005}{alg-geom/9604005}.  

\bibitem{Tana} Y. Tanaka, {\it On the singular sets of solutions to the Kapustin--Witten equations and the Vafa--Witten ones on compact K\"ahler surfaces}, Geom. Dedicata 199  (2019), 177--187. \href{http://arxiv.org/abs/1510.07739}{arXiv:1510.07739}. 

\bibitem{Taub1} C. H. Taubes, {\it $PSL(2;\C)$-connections on 3-manifolds with $L^2$ bounds on curvature}, 
Camb. J. Math.  1 (2013), 239--397.  \href{https://arxiv.org/abs/1205.0514}{arXiv:1205.0514}. 

\bibitem{Taub2} C. H. Taubes, {\it Compactness theorems for $SL(2;\C)$ generalizations of the 4-dimensional anti-self-dual equations}. \href{https://arxiv.org/abs/1307.6447}{arXiv:1307.6447}. 2013. 

\bibitem{UhYa} K. Uhlenbeck and S-T. Yau, {\it On the existence of Hermitian-Yang--Mills connections in stable vector bundles},  
Frontiers of the mathematical sciences: 1985 (New York, 1985). Comm. Pure Appl. Math. 39 (1986),  no. S, suppl., S257--S293. 

\bibitem{Ward} R. S. Ward, {\it Integrable $(2k)$-dimensional Hitchin equations}, Lett. Math. Phys. 106 (2016),  951--958. \href{https://arxiv.org/abs/1604.07247}{arXiv:1604.07247}. 

\bibitem{Witt} E. Witten, {\it Fivebranes and knots}, Quantum Topol. 3 (2012),  1--137. 
\href{https://arxiv.org/abs/1101.3216}{arXiv:1101.3216}. 

\end{thebibliography}
\end{document}